\documentclass{jnmp}

\usepackage{graphicx}
\usepackage{amsmath}
\usepackage{amssymb}
\usepackage{amsfonts}
\usepackage{enumerate}
\usepackage[utf8]{inputenc}

		\theoremstyle{plain}
      \newtheorem{theorem}{Theorem}[section]
      \newtheorem{lemma}[theorem]{Lemma}
      \newtheorem{proposition}[theorem]{Proposition}     
      \newtheorem{corollary}[theorem]{Corollary}
      \theoremstyle{definition}
      \newtheorem{definition}[theorem]{Definition}       
      \newtheorem{example}[theorem]{Example}     
      \theoremstyle{remark}
      \newtheorem{remark}[theorem]{Remark}
      \newtheorem*{acknowledgments}{Acknowledgments}
      
\JNMPnumberwithin{equation}{section}
\resetfootnoterule
      
\begin{document}

\renewcommand{\evenhead}{J A Vallejo}
\renewcommand{\oddhead}{Euler-Lagrange equations for functionals defined on Fréchet manifolds}

\thispagestyle{empty}


\copyrightnote{2009}{ J A Vallejo}

\Name{Euler-Lagrange equations for functionals defined on Fréchet manifolds}

\label{firstpage}

\Author{José Antonio VALLEJO}
\Address{     Facultad de Ciencias \\
              Universidad Autónoma de San Luis Potosí \\
              Lat. Av. Salvador Nava s/n CP78290 \\
              San Luis Potosí (SLP) México \\
              Tel.: +52-444-8262486 ext.2926 \\               
				  E-mail: jvallejo@fciencias.uaslp.mx}


\begin{abstract}
\noindent We prove a version of the variational Euler-Lagrange equations valid for functionals
defined on Fréchet manifolds, such as the spaces of sections of differentiable vector
bundles appearing in various physical theories.
\end{abstract}


\section{Introduction}

Calculus of variations is usually applied in Physics as follows: consider
a particle moving in $\mathbb{R}^{n}$ from a point $P$ to a point
$Q$. Among all possible paths $c:\mathbb{R}\rightarrow\mathbb{R}^{n}$,
the particle will follow the one that minimizes the action functional
\begin{equation}
A[c]=\int_{t_{P}}^{t_{Q}}L(c(t),\dot{c}(t))dt,
\label{eq:1}
\end{equation}
where $L\in \mathcal{C}^{\infty}(\mathbb{R}^{n}\times\mathbb{R}^{n})$
is called the Lagrange function or Lagrangian, subject to the end
points conditions
\begin{equation}
c(t_{P})=P,c(t_{Q})=Q.\label{eq:2}
\end{equation}
This path can be characterized as the solution (with initial conditions
(\ref{eq:2})) to the Euler-Lagrange equations
\begin{equation}
\frac{\partial L}{\partial\dot{x}^{i}}(c(t),\dot{c}(t))-
\frac{d}{dt}\frac{\partial L}{\partial\dot{x}^{i}}(c(t),\dot{c}(t))=0,\:1\leq i\leq n,
\label{eq:3}
\end{equation}
with $\{x^{i},y^{i}\}_{i=1}^{n}$ a set of coordinates on 
$\mathbb{R}^{n}\times\mathbb{R}^{n}$.

This is the setting of particle mechanics, but in Physics there are
more general constructions appropriate to deal with fields. A (vector)
field is a differentiable mapping $\psi:M\rightarrow\Gamma^{\infty}E$
where $M$ is a manifold (usually Minkowski's spacetime) and $E$
is some vector bundle on $M$, with $\Gamma^{\infty}E$ its space
of differentiable sections. For example, in quantum field theory $E$
is a spinor bundle on $M$ (see \cite{GMS 09,Sar 98,Wit 87}), in gauge field theory 
$E$ is a bundle of connections on a principal bundle over $M$ 
(see \cite{CMR 03,Fer 05,MV 81}), and in general relativity $E$ is the bundle of 
Riemannian metrics 
on $M$ (see \cite{HK 78,Nut 74,Sch 89}). The dynamical equations for 
these fields are the Euler-Lagrange equations for a functional of the type
(\ref{eq:1}), where $L$ depends now on the fields $\psi(x^{\mu})$
and its derivatives $\partial_{\nu}\psi(x^{\mu})$.

From the point of view of global analysis, these fields $\psi$ are
viewed as mappings $\psi:M\rightarrow N$ where $N$ is a tensor manifold,
so we are interested in studying critical points of functionals of
the form
\[
A[\psi]=\int_{M}\mathcal{L}(\psi,\partial\psi)dx
\]
where $\mathcal{L}\in T\mathcal{C}^{\infty}(M,N)$. The space of differentiable
mappings, $\mathcal{C}^{\infty}(M,N)$, can be viewed
as an infinite-dimensional manifold, modeled on Fréchet spaces if
$M$ is compact (a case%
\footnote{There are other reasons for considering Fréchet spaces. For example,
an infinite dimensional manifold is metrizable if and only if is paracompact
modeled on Fréchet spaces.%
} which encompasses situations of the most interest in Physics, such
as dynamics on a torus or a sphere, although it is possible to extend
the theory presented here to an arbitrary $M$, see e.g Example \ref{ex:1} below.)
So, we could as well consider a Fréchet infinite-dimensional manifold
$\mathcal{M}$ \emph{ab initio} and study the functionals defined
on it.

More generally, in order to deal with dynamical situations we will
consider curves $c:J\subset\mathbb{R}\rightarrow\mathcal{M}$ and
Lagrangians depending on the curve and its derivatives. Our main goal
is to obtain an expression for the Euler-Lagrange
equations in this setting as close as possible to that of the ``finite-dimensional''
calculus of variations (\ref{eq:3}). Indeed, the main result of this paper states that
if $\mathcal{L}\in C^{\infty}(T\mathcal{M},\mathbb{R})$ is such a Lagrangian,
then, a curve $c\in C^{\infty}(J,\mathcal{M})$ is critical for $\mathcal{L}$
if and only if it verifies the Euler-Lagrange equations\[
(D_{1}L)(u(t),u'(t))-\left.\frac{d}{d\xi}\right|_{\xi=t}(D_{2}L)(u(\xi),u'(\xi))=0,\]
in a local chart where $L$ and $u(t)$ are, respectively, the local
expressions for $\mathcal{L}$ and $c(t)$, and $D_{i}L$ $(i\in\{1,2\})$
are the partial derivatives of $L$ (all the notions appearing here will be explicitly
defined in Section \ref{sec:5}).

Note that this is not a trivial task, as many of the commonly used results in the
calculus of variations on, say, Banach spaces does not hold when directly translated to
the Fréchet setting.

In Section \ref{sec:1} we offer
some motivation to our problem by considering a variational description of Einstein's
equations. Then, in Section \ref{sec:2} we fix notations and list several facts from the theory 
of infinite-dimensional manifolds for reference. Sections \ref{sec:3} and \ref{sec:4}
contain some preliminary technical results that will be used later. Finally, in 
the last section we state and prove the main result.

\section{Einstein's equations as a variational system\label{sec:1}}

Let $M$ be a differential manifold and $(E,\pi,M)$ a vector bundle.
Denote by $\Gamma^{\infty}E$ the set of differentiable sections of
$E$, by $\Gamma_{K}^{\infty}E$ the set of differentiable sections
with compact support $K\subset M$, and by $\Gamma_{0}^{\infty}E$
the set of all differentiable sections with compact support. If $M$
itself is compact, then $\Gamma^{\infty}E$ with the $C^{\infty}-$topology
is a Fréchet space (that is, a locally convex topological vector space
which is metrizable and complete.) For an arbitrary $M$, if $K\subset M$
is a compact subset then $\Gamma_{K}^{\infty}E$ is a Fréchet space
contained in $\Gamma_{0}^{\infty}E$ and if we take a fundamental
sequence of compact subsets in $M$, $\{K_{i}\}_{i\in I}$ (so $M=\bigcup_{i\in I}K_{i}$),
$\Gamma_{0}^{\infty}E$ can be endowed with the inductive limit topology
and it becomes an LF-space (see \cite{Tre 06}.) We have a disjoint union
decomposition $\Gamma^{\infty}E=\bigsqcup_{s\in\Lambda}(\Gamma_{0}^{\infty}E+s)$,
where $\Lambda$ is a certain set of sections with support not necessarily
compact. A particularly interesting case is that in which $E=T^{(r,s)}M$
is a tensor manifold over $M$; its sections are denoted by $\Gamma^{\infty}(T^{(r,s)}M)$.

The sets above can be used to endow the set of mappings $\mathcal{C}^{\infty}(M,N)$
(with $M,N$ differential manifolds) with a manifold structure. To
build an atlas for $\mathcal{C}^{\infty}(M,N)$ take a Riemannian metric $g$
on $N$; then, there exists an open set $Z\subset TN$ containing
the $0-$section and such that, if we write $Z_{q}=Z\cap T_{q}N$,
$Z_{q}$ is contained in the domain in which the exponential mapping
$exp_{q}$ is a diffeomorphism. Now, given an $f\in \mathcal{C}^{\infty}(M,N)$
we construct a set $V_{f}\subset\Gamma_{0}^{\infty}(f^{*}TN)$ as\[
V_{f}=\{\zeta\in\Gamma_{0}^{\infty}(f^{*}TN):\zeta_{q}\in Z_{f(q)}\}\]
(here $f^{*}TN$ is the pullback bundle of $TN$ by $f:M\rightarrow N$).
This $V_{f}$ is an open set. We also construct a subset of $\mathcal{C}^{\infty}(M,N)$
containing $f$ through\[
U_{f}=\{\widetilde{f}:\exists K\subset M\: compact\: with\:\widetilde{f}|_{M-K}=
f|_{M-K}\: and\:\widetilde{f}(p)\in exp_{f(p)}(Z_{f(p)})\}.\]
The charts for $\mathcal{C}^{\infty}(M,N)$ are the bijections\[
\begin{array}{rcl}
\Psi_{f}:U_{f} & \rightarrow & V_{f}\\
\widetilde{f} & \mapsto & \Psi_{f}(\widetilde{f})=\zeta_{\widetilde{f}}\end{array}\]
where $\zeta_{\widetilde{f}}(p)=exp_{f(p)}^{-1}(\widetilde{f}(p))$.
It can be proved that the resulting differential structure is independent
of the chosen metric $g$ and open set $Z$.

Once an atlas of $\mathcal{C}^{\infty}(M,N)$ is known, it results that some
particularly important subsets as the immersions $\mathrm{Imm}(M,N)$,
the embeddings $\mathrm{Emb}(M,N)$, or diffeomorphisms $\mathrm{Diff}(M)$
are open, and so inherit a manifold structure.

\begin{example}
Consider the tensor bundle $S^{2}M$, of $2-$covariant symmetric
tensors on a manifold $M$. If $\mathcal{M}=S_{+}^{2}M$ denotes the
subspace of positive definite elements, then it admits a differential
manifold structure (the manifold of Riemannian metrics on $M$) modeled
on $\Gamma_{0}(S^{2}M)$ (the sections of compact support.) This structure
is such that given a metric $g\in\mathcal{M}$, its tangent space
is $T_{g}\mathcal{M}\simeq \Gamma_{0}(S^{2}M$, as $\mathcal{M}$ is open
in $S^{2}M$. A detailed study of the properties of $\mathcal{M}$,
including its Riemannian structure, can be found in \cite{GMi 91}.
\end{example}

By using this example, we can give a variational description on Einstein's
equations of General Relativity. Recall that a spacetime is a four
dimensional Lorentzian manifold $(\widetilde{N},\widetilde{g})$ and
that Einstein's vacuum equations are\begin{equation}
\rho(\widetilde{g})-\frac{1}{2}\tau(\widetilde{g})\widetilde{g}=0,\label{eq:a}\end{equation}
where $\rho$ is the Ricci tensor and $\tau$ the scalar curvature
. When $\widetilde{N}$ is globally hyperbolic, it can be viewed as
a foliation $\widetilde{N}=\bigcup_{t\in\mathbb{R}}M_{t}$ by Cauchy's
$3-$dimensional hypersurfaces%
\footnote{Connected and spacelike.%
}, all of them diffeomorphic to a manifold $M$ (see \cite{BS 03}
and references therein.) Thus, we have a curve of immersions\[
\begin{array}{rcl}
\varphi:I\subset\mathbb{R} & \rightarrow & \mathrm{Imm}(M,N)\\
t & \mapsto & \varphi(t)=\varphi_{t}\end{array},\]
with $\varphi_{t}(M)=M_{t}$ the leaves of the foliation. Each $M_{t}$
has an induced Riemannian structure, so we can also consider a curve
of metrics\begin{equation}
\begin{array}{rcl}
g:I\subset\mathbb{R} & \rightarrow & \mathcal{M}\\
t & \mapsto & g(t)\equiv g_{t}=\varphi_{t}^{*}g\end{array},\label{eq:b}\end{equation}
with $\mathcal{M}$ the manifold of Riemannian metrics of $M$. With
the aid of the $3+1$ or ADM formalism (see \cite{ADM 62}), it is possible to transfer
conditions \eqref{eq:a} on $\widetilde{g}$ to a set of conditions
for the curve \eqref{eq:b}. Without entering into details (which
will appear elsewhere, \cite{GV 09}), let us note that $\widetilde{g}$
is Ricci flat (that is, it verifies \eqref{eq:a}) if and only if
the curve $g_{t}$ satisfies\begin{equation}
g''_{t}=\Gamma_{g_{t}}^{-}(g_{t},g'_{t})-2\lambda\mathrm{Grad}^{-}S(g_{t})\label{eq:c}\end{equation}
and a set of other subsidiary conditions which do not interest us
at this point. In \eqref{eq:c}, $g'_{t}$ is the derivative of $g_{t}$
with respect to $t$, $\Gamma^{-}$ is the Christoffel symbol associated
to a certain metric $G^{-}$ on $\mathcal{M}$, and $\mathrm{Grad}^{-}$
is the gradient determined by $G^{-}$. $S(g_{t})$ is the total scalar
curvature of $g_{t}$.

Thus, a variational treatment of General Relativity can be achieved
if we write equations \eqref{eq:c} as the Euler-Lagrange's equations
for a suitable Lagrangian $\mathcal{L}\in \mathcal{C}^{\infty}(T\mathcal{M})$,
to be satisfied by the extremals of an action of the form\begin{equation}
A[g_{t}]=\int_{t_{A}}^{t_{B}}\mathcal{L}(g_{t},g'_{t})dt.\label{eq:d}\end{equation}
This correspond to the image of a $4-$dimensional universe obtained
through the time evolution of a $3-$dimensional ``slice'' $M$ in
such a way that it extremizes \eqref{eq:d}.

Note that $\mathcal{L}$ takes its arguments in the tangent bundle
of an infinite dimensional Fréchet manifold when $M$ is compact.
To find a set of equations we can call Euler-Lagrange's equations,
we must make use of the techniques of differential calculus in Fréchet
spaces, which presents some differences with the finite dimensional
case (there is no direct analog of the inverse function theorem, for
instance, see \cite{Ham 82}). Also, if we want to proceed as in the
classical calculus of variations, we must give a sense to the integral
of functionals defined on duals of Fréchet spaces%
\footnote{In the finite dimensional case, the integration takes place in the
dual of $\mathbb{R}^{n}$ which is canonically identified as an euclidean
space with $\mathbb{R}^{n}$ itself.%
}, as explained in Section \ref{sec:3}. These spaces are even ``worse'' than
the Fréchet ones (in particular, they are never Fréchet), but they
still are locally convex, a property which will suffice for our purpose.
These remarks should illustrate the differences between the finite
and infinite dimensional settings.

In this paper only a part of this programme will be presented, namely,
the deduction of Euler-Lagrange's equations for a functional of the
form \eqref{eq:d}, which is an important result \emph{per se} and
could be of interest in other situations of relevance in Physics,
as stated in the Introduction.

\section{Infinite dimensional manifolds\label{sec:2}}

For the sake of generality, we will work in the category of convenient vector spaces, which
are $c^{\infty}-$complete locally convex vector spaces (also called
bornologically complete). This setting is slightly more general than that of Fréchet
and LF-spaces, and both are equivalent when applicable. However, as stated, the class of 
convenient vector spaces is wider. For background definitions and proofs in
this section, we refer the reader to the comprehensive book by A.
Kriegl and P. W. Michor \cite{KM 97}.

\begin{definition}
Let $E$ be a locally convex vector space. A net $(x_{\gamma})_{\gamma\in\Gamma}\subset E$
is called Mackey-Cauchy if there exists a bounded (absolutely convex)
set $B\subset E$ and a net 
$(\mu_{(\gamma,\gamma')})_{(\gamma,\gamma')\in\Gamma\times\Gamma}\subset\mathbb{R}$
converging to $0\in\mathbb{R}$ such that 
$x_{\gamma}-x_{\gamma'}\in\mu_{(\gamma,\gamma')}B$.
The space $E$ is Mackey complete if every Mackey-Cauchy sequence
in $E$ converges.
\end{definition}
Obviously, Mackey-Cauchy sequences are Cauchy, so sequentially complete
locally convex spaces are Mackey complete.

\begin{proposition}
If $E$ is metrizable, sequential completeness and Mackey completeness
are equivalent. Moreover, Mackey completeness of $E$ is equivalent
to the following condition: each curve $c:\mathbb{R\rightarrow E}$
with the property that for each linear mapping $l:E\rightarrow\mathbb{R}$
then $l\circ c\in\mathcal{C}^{\infty}(\mathbb{R},\mathbb{R})$, belongs
to the class $\mathcal{C}^{\infty}$.
\end{proposition}
Note that curves defined on a locally convex vector space have a natural
definition of differentiability. The spaces in which these equivalent
properties hold are the convenient vector spaces in the terminology
of Kriegl and Michor.

\begin{example}
Given a compact finite-dimensional differential manifold $M$ (not
necessarily compact) and a vector bundle $\pi:E\rightarrow M$, consider
a compact subset $K\subset M$. Then, the space of differentiable
sections with support on $K$, $\Gamma_{K}^{\infty}(E)$, is a Fréchet
space contained in $\Gamma_{0}^{\infty}(E)$, the space of differentiable
sections with compact support. If we take a fundamental sequence of
compacts in $M$, $\{K_{i}\}_{i=1}^{\infty}$, then $\Gamma_{0}^{\infty}(E)$
can be endowed with the inductive limit topology and is a convenient
vector space.
\label{ex:1}
\end{example}
Convenient vector spaces carry on the so called $c^{\infty}-$topology,
which is the final topology with respect to all smooth curves 
$c:\mathbb{R}\rightarrow E$.
Its open sets are called $c^{\infty}-$open sets.

Now, $C^{\infty}-$manifolds (of arbitrary dimensions) can be defined
by modeling them on convenient vector spaces and requiring that the
transition functions 
$u_{\alpha\beta}:u_{\beta}(U_{\alpha}\cap U_{\beta})
\rightarrow
u_{\alpha}(U_{\alpha}\cap U_{\beta})$
be smooth mappings\footnote{We use the symbol $C^{\infty}$ to denote the
class of such smooth mappings, reserving $\mathcal{C}^{\infty}$ for smooth mappings between
finite dimensional spaces} between $c^{\infty}-$open sets of a convenient
vector space (which means that these $u_{\alpha\beta}$ map smooth
curves in $u_{\beta}(U_{\alpha}\cap U_{\beta})$ to smooth curves
in $u_{\alpha}(U_{\alpha}\cap U_{\beta})$). A particular example has been considered
in the previous section, where the model space was Fréchet. Other constructions work
exactly as in the finite dimensional case.

\begin{definition}
A mapping $f:\mathcal{M}\rightarrow\mathcal{N}$ between manifolds
modeled on convenient vector spaces is smooth 
(denoted $f\in C^{\infty}(\mathcal{M},\mathcal{N})$)
if, for each $x\in\mathcal{M}$ and each chart $(\widetilde{U},v)$
around $f(x)\in\mathcal{N}$, there is a chart $(U,u)$ on $\mathcal{M}$
around $x$ such that $f(U)\subset\widetilde{U}$ and 
$v\circ f\circ u^{-1}:U\rightarrow\widetilde{U}$
is smooth as a mapping between $c^{\infty}-$open sets in convenient
vector spaces.
\end{definition}
Tangent spaces and tangent mappings are defined as usual.

\begin{example}
Let $E$ be a convenient vector space and $U\subset E$ a $c^{\infty}-$open
set. Then, viewed as an open submanifold (with the chart given by
the inclusion $j:U\hookrightarrow E$) at any $x\in U$ we have $T_{x}U\simeq E$.
\end{example}
In particular, tangent vectors can be viewed as tangents to smooth
curves as well as derivations over evaluation morphisms.

\begin{proposition}\label{pro:2}
If $f:\mathcal{M}\rightarrow\mathcal{N}$ is a smooth mapping between
manifolds modeled on convenient vector spaces, it induces a mapping
$T_{x}f:T_{x}\mathcal{M}\rightarrow T_{f(x)}\mathcal{N}$ for each
$x\in\mathcal{M}$ through the formula
\[
(T_{x}f(A))(h)=A(h\circ f)
\]
for $A\in T_{x}\mathcal{M}$, $h\in C^{\infty}(\mathcal{N},\mathbb{R})$.
\end{proposition}
For an infinite-dimensional manifold $\mathcal{M}$, an analogue of
the exponential mapping is introduced.

\begin{definition}
Let $T\mathcal{M}$ be the tangent bundle of $\mathcal{M}$, with
projection $\pi_{\mathcal{M}}$. Then, a mapping 
$\alpha:U\subset T\mathcal{M}\rightarrow\mathcal{M}$
defined on an open neighborhood of the zero section in $T\mathcal{M}$,
which satisfies:
\begin{enumerate}[(i)]
\item $(\pi_{\mathcal{M}},\alpha):U\subset T\mathcal{M}\rightarrow\mathcal{M}\times\mathcal{M}$
is a diffeomorphism onto a $c^{\infty}-$open neighborhood of the
origin,
\item $\alpha(0_{x})=x$ for all $x\in\mathcal{M}$,
\end{enumerate}
is called a local addition on $\mathcal{M}$.
\end{definition}

\begin{example}
\mbox{}
\begin{enumerate}[(a)]
\item The affine structure on each convenient vector space gives a local
addition.
\item Let $\pi:V\rightarrow M$ be a finite rank vector bundle over a finite
dimensional differential manifold $M$. Then, the space of sections
$\Gamma_{0}^{\infty}(V)$ admits a local addition. Moreover, $\Gamma_{0}^{\infty}(V)$
can be viewed as a splitting submanifold of $\mathcal{C}^{\infty}(M,V)$.
Even more, the total space $V$ admits a local addition.
\end{enumerate}
\end{example}

\begin{remark}
When $\mathcal{M}$ is an infinite-dimensional manifold admitting
a local addition, we have that the space $C^{\infty}(J,\mathcal{M})$,
where $J\subset\mathbb{R}$ is open, is a smooth manifold modeled
on spaces $\Gamma_{0}^{\infty}(c^{*}T\mathcal{M})$ of smooth sections
(with compact support) of pullback bundles along curves $c:J\rightarrow \mathcal{M}$.
In view of the preceding example, this is the case when $\mathcal{M}$
is a space of sections of a tensorial bundle over some compact finite-dimensional
manifold $M$, already discussed in the Introduction.
\end{remark}

\section{Integration on duals of Fréchet spaces\label{sec:3}}

In this section, we refer the reader to the text \cite{Rud 91} for definitions
and proofs unless otherwise explicitly stated.

Let $V$ be a real topological vector space. We denote by $V^{*}$
the space of continuous linear functionals $l:V\rightarrow\mathbb{R}$
and call it the dual space of $V$. This dual is a real vector space
too and it can be topologized in a variety of ways. In particular,
the weak$^{*}$ topology makes it a locally convex topological vector
space.

Recall that a family $\mathcal{F}=\{l_{\alpha}\}_{\alpha\in\Lambda}\subset V^{*}$
is said to separate points on $V$ if, for each $v_{1},v_{2}\in V$
such that $v_{1}\neq v_{2}$, there exists an $l_{\alpha}\in\mathcal{F}$
such that\[
l_{\alpha}(v_{1})\neq l_{\alpha}(v_{2}).\]
The following is an easy consequence of Hahn-Banach's theorem(s).

\begin{lemma}
\label{lem:1}If $V$ is a locally convex space then $V^{*}$ separates
points.
\end{lemma}
We also recall the construction of a certain notion of integral for
vector valued functions.

\begin{definition}
\label{def:4}Let $\mu$ be a measure on a measure space $X$, $V$
a topological vector space such that $V^{*}$ separates points, and
$f:X\rightarrow V$ a function such that the scalar functions $l^{*}f$,
for every $l\in V^{*}$, are integrable with respect to $\mu$ (where
$l^{*}f:X\rightarrow\mathbb{R}$ is defined by $(l^{*}f)(x)=l(f(x))$.)
If there exists a vector $v\in V$ such that, for every $l^{*}\in V$\[
l(v)=\int_{X}(l^{*}f)d\mu,\]
then we define\[
\int_{X}fd\mu=v.\]
\end{definition}

This integral exists under very general conditions.

\begin{proposition}
\label{pro:3}Let $V$ be a topological vector space. Suppose that:
\begin{enumerate}[(i)]
\item $V^{*}$ separates points on $V$.
\item $\mu$ is a Borel probability measure on a compact Hausdorff space
$X$\label{enu:2}.
\end{enumerate}
If $f:X\rightarrow V$ is continuous and the closed convex hull $\overline{co}(f(X))$
is compact in $V$, then the integral\[
v=\int_{X}fd\mu\]
exists in the sense of Definition \ref{def:4}.
\end{proposition}

We wish to adapt this construction to the setting we will use later,
in which $V=E^{*}$ is the dual of a Fréchet space $E$, and $X=[a,b]$
is a compact interval in the real line $\mathbb{R}$. We already know
(by Lemma \ref{lem:1}) that the dual of $E^{*}$ separates points
on $E^{*}$. Also, the Lebesgue measure on $[a,b]$ obviously satisfies
condition (\ref{enu:2}) in Proposition \ref{pro:3}. Thus, what we
need is to prove that given a continuous $f:[a,b]\rightarrow E^{*}$,
the closed convex hull $\overline{co}(f([a,b]))$ is compact in $E^{*}$.

\begin{proposition}
\label{pro:4}Let $[a,b]\subset\mathbb{R}$ and $f:[a,b]\rightarrow E^{*}$
be a continuous function, with $E$ a Fréchet space. Then, $\overline{co}(f([a,b]))$
is compact in $E^{*}$.
\end{proposition}
\begin{proof}
As $[a,b]$ is compact and $f$ continuous, $K=f([a,b])$ is compact
in $E^{*}$. Now, by Krein's theorem (see \cite{Kot 69} $\S$24.5(5)), the
closed convex hull $\overline{co}(K)$ of the compact set $K$ in
the locally convex space $E^{*}$ is compact if and only if $\overline{co}(K)$
is $\tau_{k}-$complete, where $\tau_{k}$ is the Mackey topology
on $E^{*}$. But $E$ is Fréchet, so $E^{*}$ is automatically $\tau_{k}-$complete
(see \cite{Kot 69} $\S$21.6(4)), so the closed set $\overline{co}(K)$ is also
$\tau_{k}-$complete and thus compact.
\end{proof}
\begin{corollary}
Let $[a,b]\subset\mathbb{R}$ and $f:[a,b]\rightarrow E^{*}$
be a continuous function with $E$ a Fréchet space. Then $l=\int_{[a,b]}fdt$
exists as a vector integral and $l\in E^{*}$. For every $e\in E$,\[
l(e)=\int_{[a,b]}(f(t))(e)dt.\]
\label{cor:1}
\end{corollary}
\begin{proof}
Immediate from Proposition \ref{pro:4} and Definition \ref{def:4}.
\end{proof}
\begin{remark}
Of course, the same construction of the integral works for functions
$f:[a,b]\rightarrow E$, with $E$ Fréchet. In this case the situation
is even simpler, as for a Fréchet space it is well known that if $K\subset E$
si compact, then $\overline{co}(K)$ is compact too.
\end{remark}

\section{DuBois-Reymond Lemma\label{sec:4}}

In this section, we generalize the classical DuBois-Reymond lemma
from the theory of the calculus of variations, for later use.

\begin{lemma}
\label{lem:2}Let $E$ be a Fréchet space and $f:J\subset\mathbb{R}\rightarrow E^{*}$
a continuous mapping with $J=[a,b]$. Then\[
\int_{J}f(t)(\mu'(t))dt=0,\]
for all $\mu\in C^{1}(J,E)$ with support $\sup\mu\subset J$, if
and only if $f$ is constant.
\end{lemma}
\begin{proof}
For the nontrivial implication, note that if $\phi\in C^{1}(J,\mathbb{R})$
and $e\in E$ is a fixed element, then the mapping $\mu:J\rightarrow E$
given by $\mu(t)=y\phi(t)$ verifies $\mu\in C^{1}(J,E)$ (and has
support $\sup\mu\subset J$ if $\sup\phi\subset J$). Thus, for any
fixed $l\in E^{*}$, $y\in E$, $\phi\in C^{1}(J,\mathbb{R})$,\[
\int_{J}(f(t)-l)(y\phi'(t))dt=0.\]
Take now\[
l=\int_{J}f(t)dt\]
(this integral in the sense of the preceding section) and\[
\phi(t)=\int_{a}^{t}(f(s)-l)(y)ds.\]
We have $\phi\in C^{1}(J,\mathbb{R})$ and, applying Corollary \ref{cor:1}
and the hypothesis,\[
\int_{J}((f(t)-l)(y))^{2}dt=0,\]
so $f(t)=l$, for all $t\in J$.
\end{proof}
\begin{proposition}[DuBois-Reymond Lemma]
\label{pro:5}
Let $f,g:J\subset\mathbb{R}\rightarrow E^{*}$ be continuous functions
with $J=[a,b]$. Then\[
\int_{J}\left\{ f(t)(\mu(t))+g(t))(\mu'(t))\right\} dt=0,\]
for all $\mu\in C^{1}(J,E)$ with support in $J$, if and only if
the mapping $h:J\rightarrow E^{*}$ given by\[
h(t)=g(t)-\int_{a}^{t}f(s)ds\]
is constant on $J$, that is, $f(t)=-g'(t)$.
\end{proposition}
\begin{proof}
Define\[
F(t)=\int_{a}^{t}f(s)ds\;(t\in J)\]
so that $f(t)=F'(t)$ and the statement we want to prove is equivalent
to\[
\int_{J}\left\{ F'(t)(\mu(t))+g(t)(\mu'(t))\right\} dt=0.\]
Now, noting that $\frac{d}{dt}\left(t\mapsto(F(t)(\mu(t))\right)=(F'(t))(\mu(t))+(F(t))(\mu'(t))$
and integrating by parts in the first member of the integrand (taking into account 
that $\sup\mu\subset J$),\[
\int_{J}\left\{ -F(t)(\mu'(t))+g(t)(\mu'(t))\right\} dt=0,\]
and the statement follows applying Lemma \ref{lem:2}.
\end{proof}

\section{The Euler-Lagrange equations\label{sec:5}}

Given a Fréchet manifold $\mathcal{M}$, there is defined a canonical
lifting of curves to the tangent bundle $\lambda:C^{\infty}(J,\mathcal{M})\rightarrow C^{\infty}(J,T\mathcal{M})$,
where $J\subset\mathbb{R}$. For each $\mathcal{L}\in C^{\infty}(T\mathcal{M},\mathbb{R})$
we will denote by $\overline{\mathcal{L}}:C^{\infty}(J,T\mathcal{M})\rightarrow\mathcal{C}^{\infty}(J)$
the mapping\[
\overline{\mathcal{L}}(\gamma)=\mathcal{L}\circ\gamma.\]

\begin{definition}
A smooth function $\mathcal{L}\in C^{\infty}(T\mathcal{M},\mathbb{R})$
is called a Lagrangian. For each open set $J\subset\mathbb{R}$ with
compact closure, we can construct the action functional\[
\begin{array}{rcl}
F_{\mathcal{L}}^{J}:C^{\infty}(J,\mathcal{M}) & \rightarrow & \mathbb{R}\\
c & \mapsto & \int_{J}(f_{\mathcal{L}}c)(t)dt\end{array},\]
where the action density associated to $\mathcal{L}$, $f_{\mathcal{L}}$,
is given by\[
f_{\mathcal{L}}=\overline{\mathcal{L}}\circ\lambda\]
(note that $f_{\mathcal{L}}c:J\rightarrow\mathbb{R}$ and, indeed,
$f_{\mathcal{L}}\in\mathcal{C}^{\infty}(J)$).
\end{definition}
In view of the results of O. Gil-Medrano (see \cite{GM 01}, Proposition 3), we define 
critical points of a density action as follows.

\begin{definition}
A curve $c\in C^{\infty}(J,\mathcal{M})$ is critical for the Lagrangian
$\mathcal{L}$ (or for the action density $f_{\mathcal{L}}$) if and
only if for each open set $J\subset\mathbb{R}$ with compact closure,\[
T_{c}F_{\mathcal{L}}^{J}(A)=0,\]
for all $A\in\Gamma^{\infty}(c^{*}T\mathcal{M})$ with support contained
in $J$.
\end{definition}
Now, we can state and prove the main result of the paper. Recall that if $U$ is a $c^{\infty}-$open set of a Fréchet space $E$, the partial derivatives of a mapping
$L:U\times E\rightarrow\mathbb{R}$, are defined by\[
D_{1}L(u,e)(f)=\lim_{\xi\rightarrow0}\frac{1}{\xi}\left(L(u+\xi f,e)-L(u,e)\right)\]
and\[
D_{2}L(u,e)(f)=\lim_{\xi\rightarrow0}\frac{1}{\xi}\left(L(u,e+\xi f)-L(u,e)\right),\]
for $(u,e)\in U\times E$ and $f\in E$.

\begin{theorem}
Let $\mathcal{L}\in C^{\infty}(T\mathcal{M},\mathbb{R})$ be a Lagrangian.
Then a curve $c\in C^{\infty}(J,\mathcal{M})$ is critical for $\mathcal{L}$
if and only if it verifies the Euler-Lagrange equations\[
(D_{1}L)(u(t),u'(t))-\left.\frac{d}{d\xi}\right|_{\xi=t}(D_{2}L)(u(\xi),u'(\xi))=0,\]
in a local chart where $L$ and $u(t)$ are, respectively, the local
expressions for $\mathcal{L}$ and $c(t)$, and $D_{i}L$ $(i\in\{1,2\})$
are the partial derivatives of $L$.
\end{theorem}
\begin{proof}
Let $\mathcal{M}$ be a differential Fréchet manifold modeled on the
Fréchet space $E$ and let $\varphi:\mathcal{U}\subset\mathcal{M}\rightarrow U\subset E$
be a local chart, with $U$ a $c^{\infty}-$open set in $E$. Let
$c\in C^{\infty}(J,\mathcal{M})$ be a critical point of $\mathcal{L}$
and suppose (for simplicity) that $c(J)\subset\mathcal{U}$. Finally,
let $A\in\Gamma^{\infty}(c^{*}T\mathcal{M})$ be such that $\sup A\subset J$.
We have that $A$ is of the form\[
A=\left.\frac{d}{ds}\right|_{s=0}\alpha_{s}\]
for some curve\[
\begin{array}{rcl}
\alpha:\mathbb{R} & \rightarrow & C^{\infty}(J,\mathcal{M})\\
s & \mapsto & \alpha_{s}\end{array}\]
such that $\alpha_{0}=c:J\rightarrow\mathcal{M}$. Then (recall Proposition
\ref{pro:2})\[
0=T_{c}F_{\mathcal{L}}^{J}(A)=\left.\frac{d}{ds}\right|_{s=0}\left(s\mapsto\int_{J}(f_{\mathcal{L}}\alpha_{s})(t)dt\right).\]
Applying the rule for differentiation under the integral to the $\mathcal{C}^{\infty}$
integrand $f_{\mathcal{L}}\alpha:\mathbb{R}\times\mathbb{R}\mathbb{\rightarrow R}$
(where $(f_{\mathcal{L}}\alpha)(s,t)=(f_{\mathcal{L}}\alpha_{s})(t)$),
we get\begin{equation}
0=T_{c}F_{\mathcal{L}}^{J}(A)=\int_{J}\left(s\mapsto\left.\frac{\partial}{\partial s}\right|_{s=0}(f_{\mathcal{L}}\alpha_{s})(t)\right)dt.\label{eq:4}\end{equation}
Let us write $\varphi\circ\alpha_{s}=u_{s}:\mathbb{R}\rightarrow U$
for the local representatives of the curves $\alpha_{s}$, so $u_{s}\in C^{\infty}(J,U)$.
In particular, we will put $u=u_{0}=\varphi\circ c$. Also, let us
write $L$ for the local representative of $\mathcal{L}$ in the chart
$(\mathcal{U},\varphi)$ of $\mathcal{M}$ (and the identity on $\mathbb{R}$),
so $L=id\circ\mathcal{L}\circ(T\varphi)^{-1}$. Then (\ref{eq:4})
can be written locally as\begin{equation}
0=T_{c}F_{\mathcal{L}}^{J}(A)=\int_{J}\left(s\mapsto\left.\frac{\partial}{\partial s}\right|_{s=0}L(u_{s}(t),u'_{s}(t))\right)dt,\label{eq:5}\end{equation}
where\[
u'_{s}(t)=\frac{\partial}{\partial t}u_{s}(t).\]
Note that our setting can be represented as the integration of the
derivative evaluated at zero with respect to the second argument of
a composite map\[
\mathbb{R}\times\mathbb{R}\overset{(u,u')}{\longrightarrow}U\times E\subset E\times E\overset{L}{\longrightarrow}\mathbb{R},\]
where $U$ is a $c^{\infty}-$open set of a Fréchet space $E$. Applying
the chain rule (see Theorem 3.18 in \cite{KM 97} and Theorem 3.3.4
in \cite{Ham 82}), we get\begin{eqnarray}
0 &=& T_{c}F_{\mathcal{L}}^{J}(A)\label{eq:6}\\
  &=& \nonumber \int_{J}\left\{ D_{1}L(u(t),u'(t))\left(\left.\frac{\partial u_{s}(t)}{\partial s}\right|_{s=0}\right)+D_{2}L(u(t),u'(t))\left(\left.\frac{\partial^{2}u_{s}(t)}{\partial s\partial dt}\right|_{s=0}\right)\right\} dt
\end{eqnarray}
where $D_{1}L$ and $D_{2}L$ are the partial derivatives of $L:U\times E\rightarrow\mathbb{R}$. Note that, under the assumption
$L\in C^{\infty}(U\times E,\mathbb{R}$), the function $L$ has total
derivative $dL$ and, in this case,\[
dL(e,u)(f,g)=(D_{1}L)(e,u)(f)+(D_{2}L)(e,u)(g).\]
Moreover, we have linearity in the argument $f\in E$. Thus, we can
interpret the expression (\ref{eq:6}) as an integral of the form\[
\int_{J}\left\{ f(t)(\mu(t))+g(t))(\mu'(t))\right\} dt=0,\]
where $f(t)=D_{1}L(u(t),u'(t))$, $g(t)=D_{2}L(u(t),u'(t))$ and $\mu(t)=\left.\frac{\partial u_{s}(t)}{\partial s}\right|_{s=0}$.
Applying the DuBois-Reymond Lemma (Proposition \ref{pro:5}), we get the desired
result.
\end{proof}

\begin{acknowledgments}
The author wants to thank Olga Gil-Medrano for useful discussions 
and guidance while doing this work. Thanks are also due to the anonymous referee
whose comments helped to organize the presentation. \\
Partially supported by the Ministerio
de Educaci\'on y Ciencia of Spain, under grant code MTM$2005-04947$ and
a SEP-CONACyT project CB (J2) 2007-1 code 78791 (México).
\end{acknowledgments}

\label{lastpage}

\begin{thebibliography}{}

\bibitem{ADM 62}
\textsc{R. Arnowitt}, \textsc{S. Deser}, and  \textsc{C. W. Misner}: The Dynamics of General Relativity. In \emph{Gravitation: An introduction to current research}. L. Witten, editor. Wiley, NY (1962). Reprinted as arXiv:gr-qc/0405109.

\bibitem{BS 03}
\textsc{A. N. Bernal} and \textsc{M. Sánchez}: On Smooth Cauchy Hypersurfaces and Geroch’s Splitting Theorem.
\emph{Comm. in Math. Phys.} {\bf 243}, $461-470$ (2003).

\bibitem{CMR 03}
\textsc{M. Castrillón-López}, \textsc{J. Muñoz-Masqué} and \textsc{T. Ratiu}: Gauge invariance and 
variational trivial problems on the bundle of connections. \emph{Diff. Geom. and its appl.}
{\bf 19}, n2, $127-145$ (2003).

\bibitem{Fer 05}
\textsc{R. Ferreiro-Pérez}: Equivariant characteristic classes in the bundle of connections.
\emph{J. Geom. Phys.} {\bf 54}, $197-212$ (2005).

\bibitem{GMS 09}
\textsc{G. Giachetta}, \textsc{L. Mangiarotti} and \textsc{G. Sardanahvily}: Advanced classical field theory.
World Scientific, Singapore (2009).

\bibitem{GM 01}
\textsc{O. Gil-Medrano}: Relationship between volume and energy of vector fields.
\emph{Diff. Geom. and Appl.} {\bf 15}, n2, $137-152$ (2001).

\bibitem{GMi 91}
\textsc{O. Gil-Medrano} and \textsc{P. W. Michor}: The Riemannian manifold of all Riemannian metrics.
\emph{Oxford Quart. J. of Math.} {\bf 42}, n1, $183-202$ (1991).

\bibitem{GV 09}
\textsc{O. Gil-Medrano} and \textsc{J. A. Vallejo}: In preparation.

\bibitem{Ham 82}
\textsc{R. Hamilton}: The inverse function theorem of Nash and Moser. \emph{Bull. Amer. Math. 
Soc.} {\bf 7}, $65-222$ (1982).

\bibitem{HK 78}
\textsc{F. W. Hehl} and \textsc{G. D. Kerlick}: Metric affine variational principles in general 
relativity I:
Riemannian space-time. \emph{Gen. Rel. and Gravitation} {\bf 9}, n8, $691-710$ (1978).

\bibitem{Kot 69}
\textsc{G. Köthe}: Topological Vector Spaces I. Springer-Verlag, Berlin, Heidelberg and New 
York (1969).

\bibitem{KM 97}
\textsc{A. Kriegl} and \textsc{P. W. Michor}: The convenient setting of global analysis. AMS
Publishing, Providence (RI) (1997).

\bibitem{MV 81}
\textsc{P. K. Mitter} and \textsc{C. M. Viallet}: On the bundle of connections and the gauge orbit
method in Yang-Mills theory. \emph{Commun. Math. Phys.} {\bf 79}, $457-472$ (1981).

\bibitem{Nut 74}
\textsc{Y. Nutku}: Geometry of dynamics in General Relativity. \emph{Ann. Inst. Henri Poincaré}
{\bf XXI}, n2, $175-183$ (1974).

\bibitem{Rud 91}
\textsc{W. Rudin}: Functional analysis (2nd ed.) McGraw-Hill, New York (1991).

\bibitem{Sar 98}
\textsc{G. Sardanashvily}: Covariant spin structure. \emph{J. of Math. Phys.} {\bf 39}, n9, 
$4874-4890$ (1998).

\bibitem{Sch 89}
\textsc{H. J. Schmidt}: The metric in the superspace of Riemannian metrics and its relation 
to gravity. Proc. Conf. Diff. Geom. Appl, Brno 1989, $405-411$, World Sientific. \textsc{J. 
Janyska} and \textsc{D. Krupka}, editors. Available online as arXiv:gr-qc/0109001v1.

\bibitem{Tre 06}
\textsc{F. Trèves}: Topological vector spaces, distributions and kernels. Dover, Minneola NY (2006).

\bibitem{Wit 87}
\textsc{E. Witten}: Elliptic genera and quantum field theory. \emph{Commun. Math. Phys.} {\bf 109},
$525-536$ (1987).

\end{thebibliography}
\end{document}